\documentclass[final]{siamltex}

\usepackage{amsmath,amssymb,bm,tabularx}
\usepackage{ifpdf}
\usepackage{graphicx}
\usepackage{subfigure}
\usepackage{amsfonts,graphicx,psfrag}

\renewcommand{\Re}{\mathbb{R}}

\newcommand{\weak}{\rightharpoonup}

\newcommand{\vph}{\vphantom{A^{A}_{A}}}
\newcommand{\sst}{\,\mid\,}

\newcommand{\dv}{\mathrm{div}}

\newcommand{\eqnref}[1]{(\ref{eqn#1})}

\newcommand{\bbG}{\mathbb{G}}

\newcommand{\bbS}{\mathbb{S}}

\newcommand{\bbU}{\mathbb{U}}

\newcommand{\bbZ}{\mathbb{Z}}

\newcommand{\dbar}{{\bar{d}}}

\newcommand{\ubar}{{\bar{u}}}
\newcommand{\vbar}{{\bar{v}}}

\newcommand{\xbar}{{\bar{x}}}

\newcommand{\Sbar}{{\bar{S}}}
\newcommand{\Tbar}{{\bar{T}}}

\newcommand{\ut}{\tilde{u}}

\newcommand{\Gt}{\tilde{G}}

\newcommand{\St}{\tilde{S}}

\newcommand{\calA}{{\cal A}}

\newcommand{\calF}{{\cal F}}

\newcommand{\calP}{{\cal P}}

\newcommand{\calT}{{\cal T}}

\newcommand{\norm}[1]{\| {#1} \|}
\newcommand{\Hone}{{H^1(\Omega)}}
\newcommand{\hone}[1]{\norm{#1}_\Hone}
\newcommand{\Honeo}{{H^1_0(\Omega)}}

\newcommand{\Hmone}{{H^{-1}(\Omega)}}
\newcommand{\hmone}[1]{\norm{#1}_\Hmone}
\newcommand{\Hdiv}{{H(\Omega;\dv)}}

\newcommand{\Ltwo}{{L^2(\Omega)}}
\newcommand{\ltwo}[1]{\norm{#1}_\Ltwo}
\newcommand{\Ltwoo}{{L^2(\Omega)/\Re}}
\newcommand{\ltwoo}[1]{\norm{#1}_\Ltwoo}

\newcommand{\Lp}{{L^p(\Omega)}}

\newcommand{\Lfour}{{L^4(\Omega)}}
\newcommand{\lfour}[1]{\norm{#1}_\Lfour}
\newcommand{\Lsix}{{L^6(\Omega)}}
\newcommand{\lsix}[1]{\norm{#1}_\Lsix}

\newcommand{\Gsym}{G^{sym}}
\newcommand{\Gskew}{G^{skw}}
\newcommand{\Curl}{\mathrm{Curl}}
\newcommand{\curl}{\mathrm{curl}}
\newcommand{\normS}[1]{\norm{#1}_\bbS}

\newcommand{\Lfourthirds}{{L^{4/3}(\Omega)}}
\newcommand{\lfourthirds}[1]{\norm{#1}_\Lfourthirds}
\newenvironment{theorem*}{{\bf Theorem}\em}{\rm\mbox{}}
\newtheorem{assumption}[theorem]{Assumption}
\newtheorem{notation}[theorem]{Notation}
\newtheorem{remark}{Remark}[section]

\begin{document}

\bibliographystyle{siam}

\title{Dual--Mixed Finite Element Methods for the Navier--Stokes Equations}

\author{Jason S. Howell\thanks{Department of Mathematics, College of Charleston, Charleston, SC 29424.  Email: {\tt howelljs@cofc.edu}.  This
    material is partially based upon work supported by the Center for
    Nonlinear Analysis (CNA) under the National Science Foundation
    Grant No. DMS--0635983.}  \and Noel J.
  Walkington\thanks{Department of Mathematics, Carnegie Mellon
    University, Pittsburgh, PA 15213.  Email: {\tt noelw@andrew.cmu.edu}. Supported in part by National
    Science Foundation Grants DMS--0811029 This work was also
    supported by the NSF through the Center for Nonlinear Analysis.}
}

\date{\today}
\maketitle

\begin{abstract}
  A mixed finite element method for the Navier--Stokes equations is
  introduced in which the stress is a primary variable. The
  variational formulation retains the mathematical structure of the
  Navier--Stokes equations and the classical theory extends naturally
  to this setting. Finite element spaces satisfying the associated
  inf--sup conditions are developed.
\end{abstract}


\section{Introduction}\label{sec:intro}
The focus of this work is the development of mixed finite element
schemes for the stationary Navier--Stokes equations where the fluid
stress is a primary unknown of interest. The development of a
corresponding scheme for the Stokes problem has been recently
established \cite{howa09}; however, this scheme only computes the
symmetric part of the velocity gradient, so the extension to the
Navier--Stokes equations is not direct since the convective term
involves the full gradient.  Below we propose a mixed method based
upon the usual skew--symmetric formulation of the Navier--Stokes
equations that allows for direct approximation of the stress and the
velocity gradient.  Specifically, we write the Navier--Stokes
equations as
\begin{equation}
\begin{gathered}
(1/2) (u.\nabla) u-\dv(S) = f, \\
S =\calA(\nabla u)-pI - (1/2) u \otimes u,\\
\dv(u) = 0,
\end{gathered}
\label{eqn:nse}
\end{equation}
Here $\calA(\nabla u) = \nu (\nabla u + (\nabla u)^T)$ is the ``deviatoric''
part of the stress, $pI$ is the hydrostatic stress, and the ``Bernoulli''
stress $(1/2) u \otimes u$ arises from the identity
$$
(u.\nabla)u = (1/2)(\nabla u)u + (1/2) \dv(u \otimes u)
\quad \text{ when } \quad
\dv(u)=0.
$$
The classical formulation \cite{brfo91,GiRa79,Te77} is obtained by
eliminating $S$ from equations \eqnref{:nse}. Existence of solutions
to equations \eqnref{:nse} will be established with added regularity on the stress; in particular, $\dv(S) \in \Lfourthirds^d$ and this is only
possible if $f \in \Lfourthirds^d$. This additional regularity of the stress,
and corresponding restriction on the data, is typical of mixed methods
\cite{brfo91}.

The central issue in any mixed formulation is the set of compatibility
conditions between the spaces which are typically expressed as
inf--sup conditions.  In order to focus on these issues, and minimize
peripheral technical detail, we will only consider the stationary
problem with Dirichlet boundary data and the situation where $\calA:
\Re^{d \times d} \rightarrow \Re^{d \times d}_{sym}$ is linear.
However, the extension to include other boundary conditions, maximally
monotone stress strain relations (which model viscoelastic fluids),
and the evolution problem is direct.

The rest of this paper is organized as follows. The remainder of this
section reviews related results, and the following section develops
a variational formulation of equations \eqnref{:nse} and establishes
existence of solutions. In Section \ref{sec:feapprox} finite element
approximations are studied and standard error estimates are derived.
Finite element spaces satisfying the crucial inf--sup properties are
developed in Section \ref{sec:fespace} and a numerical example is
presented in Section \ref{sec:numerics}.

\subsection{Related Results}
Traditional numerical methods for computing approximate solutions of
fluid flows are based on the primitive velocity-pressure formulation,
and (accurate) approximations of the stress must be computed via
post-processing techniques such as $L^2$ projection or Superconvergent
Patch Recovery \cite{ba76,zh96,zitato71,zizh921,zizh922,zizh923}.  In
addition to the extra computational expense, these approximations of
the fluid stress may suffer from instabilities and may not be
appropriate for fluids with a complex microstructure, such as
shear-thinning or viscoelastic fluids.

Mixed and dual--mixed formulations that include a stress-like quantity
may be found in \cite{cawazh10,cawa09,fama96,fanipa08,fanipa09}.  
In this paper we address the following issues that have arisen
in this context.
\begin{enumerate}
\item Often non--physical quantities such as the non--symmetric
  ``pseudostress'' $\sigma = \nu\nabla u-pI$ are introduced as primary
  variables \cite{cawazh10,cawa09,fama96,fanipa08,fanipa09}.
\item The constitutive relation $\calA$ is often inverted
  \cite{cawazh10,cawa09,fama96,fanipa08,fanipa09}. Closed
  form expressions for the inverse may not be available for
  fluids exhibiting complex microstructure.
\item Often the mathematical structure of the Navier--Stokes equations is
  lost; for example, the skew symmetry of the nonlinear terms. This
  gives rise to a plethora of technical issues; examples include:
  \begin{enumerate}
  \item Often the Hilbert space setting needs to be abandoned
    \cite{fanipa08,fanipa09}.
  \item The classical energy estimate may not be available
    \cite{cawazh10,cawa09,fanipa08,fanipa09}.
  \item Elementary monotonicity arguments used for existence are not
    available and alternative arguments (e.g. BRR theory \cite{brrara80})
    are required \cite{cawazh10,fanipa08,fanipa09}.
  \end{enumerate}
  For the evolutionary problem these issues can preclude long time
  existence of solutions.
\end{enumerate}

The formulation presented below is unique in the sense
that (a) the trace-free velocity gradient is a primary unknown, (b)
the pressure is eliminated by proper definition of associated function
spaces and can be recovered by a simple postprocessing calculation,
(c) the underlying problem structure allows for nonlinear constitutive
laws which will be of critical importance when approximating flows of
non-Newtonian fluids, and (d) the skew-symmetrization of the nonlinear
convective term gives straightforward proofs of existence and
uniqueness results for the continuous and discrete variational
problems.  Additionally, the underlying structure of the
scheme is related to many finite element methods for linear elasticity
with weakly-imposed stress symmetry.

\section{Variational Formulation}\label{sec:wp}
Below $\Omega\subset\Re^d$, $d=2,3$, is a bounded domain
with Lipschitz boundary, and standard notation is used for the Lebesgue and
Sobolev spaces. The pairing $(f,g)$ denotes the standard $L^2(\Omega)$
inner product for scalar, vector, and tensor functions $f$ and $g$.

\begin{assumption} There exist constants $C, \nu > 0$ such that the
  constitutive relation $\calA:\Re^{d\times d}\to\Re^{d\times
    d}_{sym}$ satisfies
  \begin{enumerate}
  \item $\calA(G)=\calA(G^{sym})$, where $G^{sym} = (1/2)(G+G^\top)$,
  \item $(\calA(G),G) \ge \nu\ltwo{G^{sym}}^2$, and 
  \item $(\calA(G),H)\le C \ltwo{G}\ltwo{H}$.
  \end{enumerate}
\end{assumption}

The dual--mixed formulation of the Navier--Stokes equations will
be posed in the spaces:\footnote{These definitions correct the spaces given in {\tt http://dx.doi.org/10.1051/m2an/2012050}.}
\begin{gather}
  \bbG = \{ G \in \Ltwo^{d \times d} \sst tr(G) = 0 \}, \nonumber \\
  \bbU = \Lfour^d, \label{eqn:spaces} \\
  \bbS = \left\{ S\in \Ltwo^{d \times d} \sst \dv(S) \in \Lfourthirds^d
  \text{ and }  \int_\Omega tr(S) = 0\right\}.
  \nonumber
\end{gather} 
Writing $G=\nabla u$, the incompressibility condition $\dv(u)=0$ becomes
$tr(G)=0$; that is, $G\in\bbG$. The Navier--Stokes equations
may then be posed as: $(G,u,S)\in \bbG\times\bbU\times\bbS$,
\begin{align}
  (\calA(G),H) - (1/2)(u\otimes u, H) - (S, H) 
  & = 0, & H & \in \bbG \nonumber \\ 
  (1/2)(G u, v) - (\dv(S), v) 
  & = (f, v), & v & \in \bbU \label{eqn:dmNS} \\ 
  (G,T) + (u, \dv(T)) 
  & = 0, & T & \in \bbS. \nonumber
\end{align}
To illustrate that this weak statement has the same structure
as the usual formulation of the Navier--Stokes equations, we introduce the
following bilinear and (skew-symmetric) trilinear forms.
\begin{definition} \label{def:abc}
With the spaces defined as in \eqnref{:spaces}
\begin{enumerate}
\item $a:(\bbG \times \bbU)^2 \rightarrow \Re$,
  $$
  a((G,u), (H,v)) =(\calA(G),H).
  $$
\item $b:\bbS \times (\bbG \times \bbU) \rightarrow \Re$,
  $$  
  b(S,(H,v)) = (S,H) + (\dv(S),v).
  $$
  $\bbZ = Ker(B^T) = \{(G,u) \in \bbG \times \bbU \sst b(T, (G,u)) = 0,
  \,\, T \in \bbS\}$.  
\item $c: (\bbG \times \bbU)^3 \rightarrow \Re$,
  $$
  c((F,w), (G,u), (H,v))
  = (1/2) \left[(Gw,v) - ((u \otimes w),H)\right]
  = (1/2) \left[(Gw,v) - (Hw,u)\right].
  $$
\end{enumerate} 
\end{definition}
The dual--mixed formulation \eqnref{:dmNS} then takes the classical form:
$((G,u), S) \in (\bbG \times \bbU) \times \bbS$,
\begin{align*}
a((G,u), (H,v)) + c((G,u),(G,u),(H,v)) - b(S, (H,v)) &= F(H,v), 
& (H,v) &\in \bbG \times \bbU \\
b(T,(G,u)) &= 0,
&T &\in \bbS.
\end{align*}

\subsection{Well--Posedness} 
In this section it is shown that the classical analysis for the mixed
formulation of the Navier--Stokes equations extends to the dual--mixed
formulation \eqnref{:dmNS}.  The following lemma originates from
\cite{ardogu84} and is useful when testing the stress with trace free
functions $H \in \bbG$.

\begin{lemma} \label{lem:trace}
  Let $d=2$ or $3$ and $\Omega \subset \Re^d$ be a bounded Lipschitz
  domain.  If $S \in \bbS$, let $S_0 = S - (1/d) tr(S)
  I$ denote the trace-free part of $S$. Then
  $$
  \ltwo{tr(S)} 
  \leq C \left( \ltwo{S_0} + \hmone{\dv(S)} \right).
  $$
  In particular,
  $$
  \ltwo{S_0}^2 + \lfourthirds{\dv(S)}^2
  \leq \normS{S}^2
  \leq C ( \ltwo{S_0}^2 + \lfourthirds{\dv(S)}^2),
  $$
  where $\normS{S}^2 \equiv \ltwo{S}^2 + \lfourthirds{\dv(S)}^2$.
\end{lemma}

In the current context the Korn and Poincar\'{e} inequalities
correspond to bounds upon $G^{skw}$ and $u$ by $G^{sym}$.

\begin{lemma} \label{lem:KornSobolev} Let the spaces $\bbG$, $\bbU$
  and $\bbS$ be the spaces characterized in equation \eqnref{:spaces},
  and let $\bbZ \subset \bbG \times \bbU$ be the kernel introduced in
  Definition \ref{def:abc}.
  \begin{enumerate}
  \item There exist constants $C$, $c > 0$ such that
    \begin{align*}
      \sup_{(G,u) \in \bbG \times \bbU} 
      \frac{(G,S) + (u, \dv(S))}{\norm{(G,u)}} &\geq c \normS{S}, 
      & S &\in \bbS \\[1ex]
      \norm{(\Gskew,u)} & \leq C \ltwo{\Gsym},
      & (G,u) &\in \bbZ.
    \end{align*}

  \item If $(G,u) \in \bbZ$ then $\lsix{u} \leq C \ltwo{G^{sym}}$;
    moreover, if $\{(G_n, u_n)\}_{n=0}^\infty \subset \bbZ$ and $G_n
    \weak G$ in $\Ltwo^{d \times d}$ then $u_n \rightarrow u$ in
    $\Lp^d$ for $1 \leq p < 6$.
  \end{enumerate}
\end{lemma}

The inf-sup condition follows directly upon selecting $(G,u) = (S_0,
\dv(S))$ and appealing to the previous lemma.  If $(G,u) \in \bbZ$,
then $u \in \Honeo$ with $G = \nabla u$ and $\dv(u) = 0$ so the
second assertion follows from the Korn and Poincar\'{e} inequalities
and the Sobolev embedding theorem.

\begin{corollary} 
  Let $a(.,.)$, $b(.,.)$ and $c(.,.,.)$ be the functions and $\bbZ$ be
  the kernel introduced in Definition \ref{def:abc}.
  \begin{enumerate}
  \item $a(.,.)$ and $b(.,.)$ are continuous and $a(.,.)$
    is coercive on $\bbZ$;
    $$
    a((G,u), (G,u)) \geq \nu \ltwo{G^{sym}}^2 
    \geq c_a \norm{(G,u)}^2
    \qquad
    (G,u) \in \bbZ,
    $$
    where $c_a=\nu/(1+C^2)$.
  \item $c:\bbZ \times \bbZ \rightarrow (\bbG \times \bbU)'$ is weakly
    continuous.
  \end{enumerate}
\end{corollary}

The following theorem establishes existence for the continuous
problem \eqnref{:dmNS}, and will also provide existence of 
solutions for the numerical scheme. 

\begin{theorem} \label{thm:existence}
  Let $\bbG$, $\bbU$ and $\bbS$ be separable reflexive Banach spaces,
  $a:(\bbG \times \bbU)^2 \rightarrow \Re$ and $b:\bbS \times (\bbG
  \times \bbU) \rightarrow \Re$ be bilinear and continuous. Let
  $$
  \bbZ = \{(G,u) \in \bbG \times \bbU \sst b(T,(G,u)) = 0,
  \,\, T \in \bbS\},
  $$ 
  and $c:\bbZ^2 \times (\bbG \times \bbU) \rightarrow \Re$ be
  trilinear and continuous.  Assume
  \begin{enumerate}
  \item $a(.,.)$ is coercive on $\bbZ$:
    $
    a((G,u), (G,u)) \geq c_a \norm{(G,u)}^2
    $
    for all $(G,u) \in \bbZ$.

  \item $b(.,.)$ satisfies the inf-sup condition
    $$
    \sup_{(G,u) \in \bbG \times \bbU} 
    \frac{b(S, (G,u))}{\norm{(G,u)}} \geq c_b \norm{S}_\bbS, 
    \qquad
    S \in \bbS.
    $$

  \item $c((G,u), (G,u), (G,u)) = 0$ for all $(G,u) \in \bbZ$,
    and the map $c:\bbZ \times \bbZ \rightarrow (\bbG \times \bbU)'$
    is weakly continuous.
  \end{enumerate}
  Then for each $F \in (\bbG \times \bbU)'$ there exists
  $(G,u,S) \in \bbG \times \bbU \times \bbS$ such that
  \begin{align*}
    a((G,u), (H,v)) + c((G,u), (G,u), (H,v)) - b(S, (H,v)) 
    &= F(H,v),
    & (H,v) & \in \bbG \times \bbU, \\
    b(T, (G,u)) &=0
    & T & \in \bbS.
  \end{align*}
  Moreover, $\norm{(G,u)} \leq (1/c_a) \norm{F}$ and $\norm{S}_\bbS \leq
  (1/c_b) \left(1 + C_a/c_a + (C_c/c_a^2) \norm{F} \vph\right) \norm{F}$
  where $C_a$ and $C_c$ denote continuity constants of
  $a(.,.)$ and $c(.,.,.)$ respectively.
\end{theorem}

\begin{proof}
  Let $\calF:\bbZ \rightarrow \bbZ$ be characterized by
  $$
  (\calF(G, u), (H, v))
   = a((G,u), (H,v)) + c((G,u), (G,u), (H,v)) - F(H,v),
   \qquad (H,v) \in \bbZ,
  $$
  where the pairing on the left is an inner product on $\bbG \times
  \bbU$ (for example, the symmetric part of $a(.,.)$). Setting $(H,v)
  = (G,u) \in \bbZ$ shows
 $$
 (\calF(G, u), (G, u))
 = a((G,u), (G,u)) - F(G,u) 
 \geq c_a \norm{(G,u)}^2 - F(G,u),
 $$
 It follows \cite[Corollary II.2.2]{Sh97} that $\calF(G,u) = 0$ for some
 $(G, u) \in \bbZ$ with norm $\norm{(G,u)} \leq (1/c_a) \norm{F}$.

 Existence of a stress follows from the continuity and coercivity of
 $b(.,.)$ on $\bbS \times (\bbG \times \bbU)/\bbZ$. Specifically,
 if $(G,u) \in \bbZ$ satisfies $\calF(G,u) = 0$, then the mapping
 $$
 (H,v) \mapsto a((G,u), (H,v)) + c((G,u), (G,u), (H,v)) - F(H,v)
 $$
 vanishes on $\bbZ$, so is in the dual of $(\bbG \times \bbU)/\bbZ$.
 It follows that the problem; $S \in \bbS$,
 $$
 b(S, (H,v)) = a((G,u), (H,v)) + c((G,u), (G,u), (H,v)) - F(H,v),
 \qquad (H,v) \in \bbG \times \bbU,
 $$
 has a unique solution; moreover
 $$
 c_b \norm{S}_\bbS 
 \leq C_a \norm{(G,u)} + C_c \norm{(G,u)}^2 + \norm{F}
 \leq (1 + C_a/c_a + C_c/c_a^2 \norm{F}) \norm{F}.
 $$
\end{proof}

\section{Finite Element Approximation}\label{sec:feapprox}

Let $\bbG \times \bbU \times \bbS$ be the spaces defined in
\eqnref{:spaces} and $\bbG_h \times \bbU_h \times \bbS_h$ be (finite
element) subspaces.  The discrete problem corresponding to the
variational form \eqnref{:dmNS} is: $(G_h, u_h, S_h) \in
\bbG_h \times \bbU_h \times \bbS_h$,
\begin{align}
  (A(G_h),H_h) - (1/2)(u_h \otimes u_h, H_h) - (S_h, H_h) 
  & = 0, & H_h & \in \bbG_h \nonumber \\ 
  (1/2)(G_h u_h, v_h) - (\dv(S_h), v_h) 
  & = (f, v_h), & v_h & \in \bbU_h \label{eqn:dmNSh} \\ 
  (G_h,T_h) + (u_h, \dv(T_h)) 
  & = 0, & T_h & \in \bbS_h. \nonumber
\end{align}
Using the functions $a(.,.)$, $b(.,.)$ and $c(.,.,.)$ in Definition
\ref{def:abc} the discrete weak problem can be written as: $(G_h, u_h,
S_h) \in \bbG_h \times \bbU_h \times \bbS_h$,
\begin{align*}
  a((G_h,u_h), (H_h,v_h)) + c((G_h,u_h), (G_h,u_h), (H_h,v_h))
  + b(S_h, (H_h,v_h)) 
  &= F(H_h,v_h), \\
  b(T_h, (G_h,u_h)) &=0,
\end{align*}
for $(H_h,v_h) \in \bbG_h \times \bbU_h$ and $T_h \in \bbS_h.$

In order for the discrete problem to be well--posed the discrete spaces
need to inherit the inf-sup and Korn/Poincar\'{e} estimates stated in Lemma
\ref{lem:KornSobolev}.

\begin{assumption} \label{ass:infsup}
There exist constants $c_b$ and $C > 0$ independent of $h$ such that
\begin{align}
  \sup_{(G_h,u_h) \in \bbG_h \times \bbU_h} 
  \frac{(G_h,S_h) + (u_h, \dv(S_h))}{\norm{(G_h,u_h)}} 
  &\geq c_b \normS{S_h}, 
  & S_h &\in \bbS_h, \label{eqn:dinfsup}\\[1ex]
  \norm{(\Gskew_h,u_h)} & \leq C \ltwo{\Gsym}
  & (G_h,u_h) &\in \bbZ_h.\label{eqn:dkorn}
\end{align}
where
$
\bbZ_h = \{ (G_h, u_h) \in \bbG_h \times \bbU_h \sst
(G_h, S_h) + (u_h, \dv(S_h)) = 0, \,\,\, S_h \in \bbS_h\}.
$
\end{assumption}

\subsection{Discrete Weak Problem} 
Existence of a solution to the discrete problem will follow from
Theorem \ref{thm:existence} whenever the discrete spaces inherit the
inf-sup condition \eqnref{:dinfsup} and discrete Korn inequality
\eqnref{:dkorn} since weak continuity of the trilinear form $c(.,.,.)$ is
immediate on finite dimensional spaces.

\begin{lemma} 
  Let $(\bbG_h, \bbU_h, \bbS_h) \subset (\bbG, \bbU, \bbS)$ be a
  finite dimensional subspace satisfying Assumption \ref{ass:infsup}.
  Then there exists $(G_h, u_h, S_h) \in (\bbG_h, \bbU_h, \bbS_h)$
  satisfying equations \eqnref{:dmNSh} such that
  $$
  \ltwo{G_h} + \lfour{u_h} \leq (1/c_a) \lfourthirds{f}.
  $$
  and
  $$
  \normS{S_h} \leq (1/c_b) 
  \left(1 + C_a/c_a + (C_c/c_a^2) \lfourthirds{f} \vph\right) \lfourthirds{f},
  $$
  where $C_c$ is the norm of $c(.,.,.)$ on $\bbZ \times
  \bbZ\times (\bbG \times \bbU)$.
\end{lemma}

\subsection{Error Estimates} 
The following analogue of Lemma \ref{lem:KornSobolev} establishes the
compactness and embedding properties of the discrete spaces necessary
to control the trilinear form $c(.,.,.)$.
\begin{lemma} \label{lem:bbZh}
  Let $\{(\bbG_h, \bbU_h, \bbS_h)\}_{h > 0}$ be a family of finite
  element subspaces of $(\bbG, \bbU, \bbS)$ constructed over a regular
  family of triangulations of $\Omega$ satisfying the hypotheses of
  Assumption \ref{ass:infsup}. 
  \begin{enumerate}
  \item If $(G_h, u_h) \in \bbZ_h$ and $G_h \weak G$ in $\Ltwo$,
    then $u_h \rightarrow u$ in $\Lfour$.
    
  \item If the triangulations are quasi-uniform there exists $C$
    independent of $h$ such that $\lsix{u_h} \leq C \ltwo{G_h^{sym}}$
    for $(G_h, u_h) \in \bbZ_h$.
  \end{enumerate}
\end{lemma}

\begin{proof}
  Fix $(G_h,u_h) \in \bbZ_h$ and
  let $(\Gt,\ut,\St) \in \bbG \times \bbU \times \bbS$ satisfy
  \begin{align*}
    (\Gt,H) - (v,\dv(\St)) - (\St,H)  
    &= (G_h, H) & (H,v) & \in \bbG \times \bbU \\
    (\ut,\dv(T)) + (\Gt,T) &= 0 & T & \in \bbS.
  \end{align*}
  Then $\ut \in \Honeo$, $\nabla \ut = \Gt$, and $\ltwo{\Gt} \leq
  \ltwo{G_h}$.  The Poincar\'{e} inequality and the Sobolev embedding theorem (in
  three dimensions) then show $\lsix{\ut} \leq C \ltwo{G_h}$. Notice
  that $(G_h,u_h,0)$ satisfies the discrete version of this equation,
  so classical finite element theory shows
  $$
  \ltwo{\ut-u_h} \leq C \ltwo{\nabla \ut} h \leq C \ltwo{G_h} h.
  $$
  If $I_h: \Hone \rightarrow \bbU_h$ denotes the Cl{\'e}ment 
  interpolant \cite{cl75}, the bound on $\lsix{u_h}$ follows from classical
  inverse estimates and approximation properties of $I_h$
  $$
  \lsix{u_h} 
  \leq \lsix{I_h \ut} + \lsix{I_h \ut - u_h} 
  \leq C \hone{I_h \ut} + (1/h) \ltwo{I_h \ut - u_h}
  \leq C \hone{\ut}.
  $$
\end{proof}

As with the Navier--Stokes equations \cite{la08}, solutions are unique
when $f$ is sufficiently small, and the discrete problem exhibits
optimal rates of convergence.

\begin{theorem} Let $\Omega\subset\Re^d$ have Lipschitz continuous
  boundary and let $\{(\bbG_h, \bbU_h, \bbS_h)\}_{h > 0}$ be a family
  of finite element subspaces of $(\bbG, \bbU, \bbS)$ constructed over
  a regular family of quasi-uniform triangulations of $\Omega$
  satisfying the hypotheses of Assumption \ref{ass:infsup}.  Assume
  $(G,u,S)\in \bbG \times \bbU \times \bbS$ satisfies equations
  \eqnref{:dmNS} and $(G_h,u_h,S_h)\in \bbG_h\times \bbU_h\times
  \bbS_h$ satisfies equations \eqnref{:dmNSh}.  If $\lfourthirds{f}$
  is sufficiently small there is a constant $C$, independent of $h$,
  such that
  \begin{multline*}
    \ltwo{G-G_h}+\lfour{u-u_h}+\normS{S-S_h} \\
    \le C \left\{\inf_{H_h\in\bbG_h}\ltwo{G-H_h}
      + \inf_{v_h\in U_h}\lfour{u-v_h}
      + \inf_{T_h\in\bbS_h}\normS{S-T_h} \right\}.
  \end{multline*}
\end{theorem}

\begin{proof} The proof is similar to the standard approach used for
  the classical mixed formulation of the Navier--Stokes equations
  \cite{la08}, and will only be sketched here.  The
  Galerkin orthogonality condition becomes
\begin{multline*}
  a\left((G-G_h,u-u_h), (H_h,v_h) \vph\right)
  + b\left(S-S_h, (H_h,v_h) \vph\right) \\
  = c\left((G_h,u_h), (G_h,u_h), (H_h,v_h) \vph\right) - 
  c\left((G,u), (G,u), (H_h,v_h) \vph\right),
\end{multline*}
for all $(H_h,v_h) \in \bbG_h \times \bbU_h$. 
Fix $(G_p, u_p) \in \bbZ_h$ and write 
$$
(E,e) \equiv (G-G_h, u-u_h) 
= (G-G_p, u-u_p) + (G_p-G_h, u_p-u_h) 
\equiv (E_p, e_p) + (E_h, e_h).
$$
Setting $(H_h,v_h) = (E_h, e_h) \in \bbZ_h$ and using the coercivity
of $a(.,.)$ on $\bbZ_h$ and skew-symmetry of $c(.,.,.)$ it follows
that
$$
c \norm{(E_h,e_h)}
\leq C \left( \norm{(E_p,e_p)} + \normS{S-T_h}
+ \norm{(G,u)} \norm{(E_h,e_h)} \right),
$$
where $T_h \in \bbS_h$ is arbitrary. When $\norm{(G,u)} \leq C \lfourthirds{f}$
is sufficiently small, the last term on the right can be absorbed into
the left to show
\begin{eqnarray*}
\norm{(E, e)}
&\leq& C \big( \inf_{(G_p,u_p) \in \bbZ_h} \norm{(G-G_p, u-u_p)} 
  + \inf_{T_h \in \bbS_h} \normS{S-T_h} \big) \\
&\leq& C \big( \inf_{(H_h,v_h) \in \bbG_h \times \bbU_h} \norm{(G-H_h, u-v_h)} 
  + \inf_{T_h \in \bbS_h} \normS{S-T_h} \big), 
\end{eqnarray*}
where the last line follows from the property that the discrete
kernels $\bbZ_h$ optimally approximate $\bbZ$ when $b(.,.)$ satisfies
the inf-sup condition.  The error estimate for the stress now follows
from the orthogonality relation and the inf-sup property of $b(.,.)$.
\end{proof}

\section{Finite Element Spaces}\label{sec:fespace}
This section considers the development of finite element subspaces
satisfying the crucial inf-sup condition in Assumption
\ref{ass:infsup}. Lemma \ref{lem:infsup} below provides several
equivalent formulations of the inf-sup condition useful for this task.
This lemma shows that if $\bbG_h \times \bbU_h \times \bbS_h$
satisfies Assumption \ref{ass:infsup}, then the space $\bbG_h^{skw}
\times \bbU_h \times \bbS_h$ is a stable space for the elasticity
problem with weak symmetry \cite{brfo91,arfawi07,boffibook}; here
$\bbG_h^{skw}$ denotes the subspace of skew matrices in $\bbG_h$.  
However, this is not sufficient; an additional property is required if
Assumption \ref{ass:infsup} is to hold. In Section \ref{sec:PEERSAFW}
it is shown that in two dimensions the finite element spaces developed
for the elasticity problem will typically inherit the additional
requirement; however, this is not so in three dimensions. This issue
is circumvented in Section \ref{sec:fespaces3d} which develops a new
family of elements satisfying Assumption \ref{ass:infsup} in two and
three dimensions.

The following lemma reformulates the inf-sup condition of
Assumption \ref{ass:infsup} into a form more amenable to analysis
using macroelement techniques.

\begin{lemma} \label{lem:infsup}
  Let $\bbG \subset \Ltwo^{d \times d}$ be closed under transpose,
  $\bbU \subset \Lfour^d$, and 
  \[\bbS \subset \left\{ S\in \Ltwo^{d \times d} \sst \dv(S) \in \Lfourthirds^d
  \text{ and }  \int_\Omega tr(S) = 0\right\},\]
  be closed subspaces.
  Let
  \begin{align}
  \bbZ &= \{(G,u) \in \bbG \times \bbU \sst 
  (G,T) + (u, \dv(T)) = 0, \,\, T \in \bbS\},\nonumber\\
    Z &= \{T \in \bbS \sst (u,\dv(T)) = 0, \,\, u \in \bbU\},\label{eqn:zdefs} \\
    Z^{sym} &= \{T \in \bbS \sst (\Gskew,T) + (u,\dv(T)) = 0, \nonumber\,\, 
    (\Gskew, u) \in \bbG^{skw} \times \bbU\},
  \end{align}
  where $\bbG = \bbG^{skw} \oplus \bbG^{sym}$ is the decomposition of $\bbG$
  into skew-symmetric and symmetric matrices.
  Then the following are equivalent:
  \begin{enumerate}
  \item There exist constants $c$ and $C > 0$ such that
    \begin{align*}
      \sup_{(G,u) \in \bbG \times \bbU} 
      \frac{(G,T) + (u, \dv(T))}{\norm{(G,u)}} &\geq c \normS{T},
      & T &\in \bbS, \\[1ex]
      \norm{(\Gskew,u)} & \leq C \ltwo{\Gsym},
      & (G,u) &\in \bbZ.
    \end{align*}

  \item There exists a constant $c > 0$ such that    
    \begin{align}
      \sup_{T \in \bbS}
      \frac{(\Gskew,T) + (u,\dv(T))}{\normS{T}} &\geq c \norm{(\Gskew,u)}, 
      &\quad (\Gskew,u) &\in \bbG^{skw} \times \bbU, \label{eqn:bSkew} \\
      \sup_{\Gsym \in \bbG^{sym}}
      \frac{(\Gsym,T)}{\ltwo{\Gsym}} &\geq c \normS{T},
      &\quad T &\in Z^{sym}. \label{eqn:bSym}
    \end{align}

  \item There exists a constant $c > 0$ such that
    \begin{align*}
      \sup_{T \in \bbS} \frac{(u,\dv(T))}{\normS{T}} &\geq c \lfour{u},
      &\qquad u &\in \bbU, \nonumber\\
      \sup_{T \in Z} \frac{(\Gskew,T)}{\normS{T}} &\geq c \ltwo{\Gskew},
      &\quad \Gskew &\in \bbG^{skw}, \nonumber\\
      \sup_{\Gsym \in \bbG^{sym}}
      \frac{(\Gsym,T)}{\ltwo{\Gsym}} &\geq c \normS{T},
      &\quad T &\in Z^{sym}.
    \end{align*}
  \end{enumerate}
\end{lemma}

\begin{remark}
  \begin{enumerate}
  \item Note that tensors in $Z^{sym}$ are only ``weakly symmetric'', i.e., they
    need not be symmetric pointwise.

  \item The first condition of (2), or equivalently the first two
    conditions of (3), are the stability conditions for the elasticity
    problem with weak symmetry. The development of stable spaces for
    this problem can be found in
    \cite{arbrdo84,arfawi07,bobrfo09,cogogu10,gogu10,st88}.
    
  \item The last condition of (2) and (3) is necessary to compute the
    full gradient, $G$. Spaces developed for the elasticity problem
    will only compute the symmetric part of the gradient if this
    condition fails.

  \item If $\dv(\bbS)\subset \bbU$ then $\normS{T} = \ltwo{T}$ for tensors $T \in Z^{sym}$.
  \end{enumerate}
\end{remark}

\begin{proof}
The hypothesis that $\bbG$ is closed under transpose allows $\bbG$
to be decomposed into a direct sum of skew and symmetric matrices,
$\bbG = \bbG^{skw} \oplus \bbG^{sym}$.
Then writing
$$
b_1(T, \Gsym) + b_2(T, (\Gskew, u))
= (\Gsym,T) + (\Gskew,T) + (u, \dv(T)),
$$
the equivalence of conditions (1) and (2) follows from the equivalence
of conditions (1) and (2) in \cite[Theorem 3.2]{howa09}.

The equivalence of conditions (1) and (3) in \cite[Theorem 3.1]{howa09}
shows that the inf-sup condition in equation \eqnref{:bSkew} is
equivalent to
\begin{align*}
\sup_{T \in \bbS} \frac{(u,\dv(T))}{\normS{T}} &\geq c \lfour{u},
&\qquad u &\in \bbU, \\
\sup_{T \in Z} \frac{(G^{skw},T)}{\normS{T}} &\geq c \ltwo{G^{skw}},
&\qquad G^{skw} &\in \bbG^{skw}.
\end{align*}
\end{proof}

\subsection{Macroelement Construction}
The following notation facilitates a unified
discussion of the two and three dimensional situation.

\begin{notation}
\begin{enumerate}
\item If $\bbG_h^{skw} \times \bbU_h \times \bbS_h$ is a subspace of
  $\bbG^{skw} \times \bbU \times \bbS$, then $\bbZ_h, Z_h,$ and
  $Z^{sym}_h$ denote the analogues of the spaces $\bbZ, Z,$ and
  $Z^{sym}$ defined in \eqnref{:zdefs}, respectively.

\item When $d=2$, if $a:\Omega \rightarrow \Re$ and $\psi:\Omega
  \rightarrow \Re^2$ then
  $$
  W(a) 
  = \begin{bmatrix} 0 & a \\ -a & 0\end{bmatrix}, 
  \qquad 
  \Curl(\psi) 
  =  \begin{bmatrix} -\psi_{1,y} & \psi_{1,x} \\ 
    -\psi_{2,y} & \psi_{2,x}\end{bmatrix}.
  $$
  If $V_h \subset \Hone$, then
  $
  \Curl(V_h) = \{ \Curl(\psi) \sst \psi \in V_h^2 \}.
  $
  
\item  When $d=3$, if $a:\Omega \rightarrow \Re^3$ and $\psi:\Omega
  \rightarrow \Re^{3 \times 3}$ then
  \begin{small}
  $$
  W(a)=\begin{bmatrix} 0 & -a_3 & a_2 \\ a_3 & 0 & -a_1 
    \\ -a_2 & a_1 & 0 \end{bmatrix},\quad
  \Curl(\psi) = \begin{bmatrix} 
    \psi_{13,y}-\psi_{12,z} & \psi_{11,z}-\psi_{13,x} 
    & \psi_{12,x}-\psi_{11,y}\\ 
    \psi_{23,y}-\psi_{22,z} & \psi_{21,z}-\psi_{23,x} 
    & \psi_{22,x}-\psi_{21,y} \\ 
    \psi_{33,y}-\psi_{32,z} & \psi_{31,z}-\psi_{33,x} 
    & \psi_{32,x}-\psi_{31,y}\end{bmatrix}. 
  $$
  \end{small}
  If $V_h \subset \Hone$, then
  $
  \Curl(V_h) = \{ \Curl(\psi) \sst \psi \in V_h^{3 \times 3} \}.
  $
\end{enumerate}
\end{notation}

If $\bbG_h^{skw} \times \bbU_h \times \bbS_h$ is a stable triple of
finite element spaces for the elasticity problem with weak symmetry,
the macroelement technique \cite{brfo91,st84} can be used to establish
the last condition in (3) of Lemma \ref{lem:infsup} by showing that
the only tensors $S_h \in Z^{sym}$ orthogonal to $\bbG_h^{sym}$ on a
macroelement take the form $S_h = \delta I$ for $\delta \in \Re$.

If $S_h \in Z_h^{sym}$ we suppose the subspace $\bbG_h^{sym}$ of
symmetric trace-free matrix valued functions is sufficiently large to
guarantee
$$
\int_M S_h : G_h^{sym} = 0, \quad G_h^{sym} \in \bbG_h^{sym}
\qquad \Rightarrow \qquad 
S_h = \delta I + W(a)
\,\,\text{ on } M,
$$
for each macroelement. If $a \equiv 0$ and $\delta \in \Re$ for tensors in
$Z^{sym}_h$ with this structure, the macroelement methodology then
shows
$$
\sup_{\Gsym_h \in \bbG^{sym}_h}
\frac{(\Gsym_h,S_h)}{\ltwo{\Gsym_h}} \geq c \ltwoo{S_h},
\qquad
S_h \in Z^{sym}_h.
$$
The following line of argument will be used for this last step.

\begin{enumerate}
\item If $n$ is the normal to a common $(d-1)$ face of two finite elements
  of $M$, the jump, $[S_h]n$, of the normal component of $S_h$
  vanishes.
  \begin{enumerate}
  \item In two dimensions 
    $$
    [S_h] n = [\delta] n - [a] n^\perp
    \quad \text{ where } \quad
    (n_1,n_2)^\perp = (-n_2,n_1).
    $$
    Since $n$ and $n^\perp$ are linearly independent it follows that
    $\delta$ and $a$ are continuous on $M$.
  \item In three dimensions
    $$
    [S_h]n = [\delta] n - [a]\times n
    $$
    It follows that $\delta$ is continuous and $[a_{tan}]=0$ (the jump
    in the tangential components of $a$ vanishes).
  \end{enumerate}

\item Tensors in $Z^{sym}_h$ are divergence free which restricts the
  jumps in the derivatives of $\delta$ and $a$.
  \begin{enumerate} 
  \item In two dimensions $\dv(S_h) = \nabla \delta - (\nabla a)^\perp$ on each
    element.  Cross differentiating shows $\Delta \delta = \Delta a = 0$ on
    each finite element of $M$,

    Also, $[\nabla \delta] = [(\nabla a)^\perp]$ along an edge
    between two finite element of $M$. Then
    $$
    [\nabla \delta].n 
    = [(\nabla a)^\perp].n 
    = \partial [a] / \partial e
    = 0,
    $$
    where $\partial [a] / \partial e$ denotes the derivative of $[a]$
    along the edge. It follows that $[\nabla \delta].n = 0$ so $\delta \in
    C^1(M)$ and similarly $a \in C^1(M)$ so $\delta$ and $a$ are
    harmonic on $M$, and hence smooth. For the usual finite element
    spaces this requires $a$ and $\delta$ each to be harmonic
    polynomials on $M$.

  \item In three dimensions, $\dv(S_h) = \nabla \delta + \curl(a)$ on
    each finite element.  Cross differentiation shows $\Delta \delta =
    \curl(\curl(a)) = 0$ on each finite element of $M$.  Also $[\nabla
    \delta] = [\curl(a)]$ on a face $k$ between two finite elements of $M$.
    Stokes' theorem shows
    $$
    \int_k -[\nabla \delta].n \, da
    = \int_k [\curl(a)].n \, da
    = \int_{\partial k} [a].ds 
    = 0,
    $$
    since $a$ is continuous at the edges (they are tangent to the
    faces).  If $\delta$ is piecewise linear then $[\nabla \delta].n =
    0$ so $\delta \in C^1(M)$ is smooth.
  \end{enumerate}
  
\item Functions in $Z^{sym}_h$ are orthogonal to $\bbG_h^{skw}$, thus
  if this space is sufficiently rich to annihilate $W(a)$ when $a$ is
  as above, conclude $a$ = 0. Then $\nabla \delta = (\nabla a)^\perp =
  0$ (2d) or $\nabla \delta = -\curl(a)=0$ (3d); and in either case
  $\delta$ is constant.
\end{enumerate}

In three dimensions the last step requires $\bbG_h^{skw}$ to
annihilate a much larger collection of (vector
valued) functions, $a$, and fails for many elements developed
for the elasticity problem with weak symmetry.

\subsection{Construction of Finite Elements}\label{sec:PEERSAFW}
In this section the macroelement methodology is used to develop finite
element triples satisfying Assumption \ref{ass:infsup}.  Two elements
will be developed for the two dimensional problem using well-known
elements for the elasticity problem with weak symmetry; counter
examples show the analogous construction fails in three dimensions.
Subsequently a new family of elements is developed which provides both
two and three dimensional elements satisfying Assumption \ref{ass:infsup}.

The following notation is adopted for the classical finite element spaces.
\begin{notation} 
Let $\{\calT_h\}_{h>0}$ be a family of triangulations of a domain
$\Omega \subset \Re^d$. 
\begin{enumerate}
\item If $M \subset \calT_h$,
  \begin{align*}
    \calP_k^{cont}(M) 
    &= \{p_h \in C(M) \sst p_h|_K \in \calP_k(K), \,\, K \in M \} 
    \quad \text{ and } \\
    \calP_k^{disc}(M)
    &= \{p_h \in L^2(M) \sst p_h|_K \in \calP_k(K), \,\, K \in M \}
  \end{align*} 
  denote the spaces of continuous and discontinuous piecewise
  polynomials of degree $k$ on $M$ respectively.  Vectors with components
  in these spaces will be denoted $\calP_k^{cont}(M)^d$ and
  $\calP_k^{disc}(M)^d$, and $d\times d$ tensors with polynomial
  components are defined similarly, and
  $$
  \calP_k^{cont}(M)^{d\times d}_{skw}
  \quad\text{ and }\quad
  \calP_k^{disc}(M)^{d\times d}_{sym}
  $$
  denote the skew and symmetric subspaces.

\item If $M \subset \calT_h$ then $RT_k(M) \subset H(\dv;M)$ and
  $BDM_k(M) \subset H(\dv;M)$ denote the subspaces of tensor valued
  functions with rows in the classical Raviart--Thomas space of order
  $k$ \cite{rath77} and Brezzi--Douglas--Marini space of degree $k$
  \cite{brdoma85}.

\item The bubble function on $\calT_h$ is denoted by $b$; this
  function is piecewise cubic when $d=2$ and quartic when $d=3$.
\end{enumerate}
\end{notation}

\subsubsection{Augmented PEERS Element} \label{sec:PEER}
In this section we augment the two dimensional PEERS
element of Arnold, Brezzi, and Douglas \cite{arbrdo84} with a
suitable class of symmetric matrices to obtain a triple satisfying
Assumption \ref{ass:infsup}.  A counterexample shows that
this construction fails in three dimensions.

\begin{lemma} Let $\calT_h$ be a triangulation of a bounded Lipschitz
  domain $\Omega \subset \Re^2$ and let
  \begin{align*}
    \bbG_h &= \bbG\cap\left(\calP_1^{cont} (\calT_h)^{2 \times 2}_{skw}
    \oplus \calP_1^{disc}(\calT_h)^{2 \times 2}_{sym} \vph\right), \\
    \bbU_h &= \calP_0^{disc}(\calT_h)^2, \\
    \bbS_h &= \bbS \cap \left( RT_0(\calT_h) 
    \oplus \calP_0^{disc}(\calT_h)^2 \otimes (\nabla b)^\perp \vph\right).
  \end{align*}
  Then the triple $\bbG_h \times \bbU_h \times \bbS_h$ satisfies 
  Assumption \ref{ass:infsup} with constant depending only upon
  the aspect ratio of $\calT_h$.
\end{lemma}

\begin{remark}
  The PEERS finite element space is $(\bbG\cap \calP_1^{cont}
  (\calT_h)^{2 \times 2}_{skw})\times \bbU_h \times \bbS_h$.
\end{remark}

\begin{proof}
  Let a typical macroelement be the set of triangles containing a
  specified vertex $x_0$ interior to $\Omega$. On each triangle the
  functions in $\bbS_h$ take the form
  $$
  S_h(x) = A + \alpha \otimes x + \psi \otimes (\nabla b(x))^\perp,
  \qquad A \in \Re^{2 \times 2}, \,\, \alpha,\psi \in \Re^2,
  $$
  and $\dv(S_h) = \alpha$. Since the
  average of $\nabla b$ vanishes on $K$, the divergence free tensors
  orthogonal to the piecewise constant trace-free matrices take the form
  $$
  S_h = \delta I + W(a) + \psi \otimes (\nabla b)^\perp,
  \qquad \delta, a \in \Re, \,\, \psi \in \Re^2.
  $$
  An elementary calculation shows that if $S_h$ is also orthogonal
  to $\bbG^{sym}_h$ then $\psi = 0$. Arguing
  as in steps (1) and (2) above shows $\delta$ and $a$ are constant functions
  on $M$. The space $\bbG^{skw}_h$ contains $W(\phi)$ where $\phi$
  is the piecewise linear ``hat'' function on $M$. Then 
  $$
  0 = \int_M W(a):W(\phi) = 2 a \int_M \phi = (2|M|/3) a,
  $$
  shows $a = 0$.
\end{proof}

The following example shows that in three dimensions the subspace
$Z^{sym}_h$ constructed from the PEERS element contains non-vanishing
skew-symmetric tensors so the inf--sup condition can not hold. This
is closely related to the property that the continuous $\calP_1^d \times 
\calP_1$ space is not div--stable.

\begin{example} 
  Given a triangulation $\calT_h$ of a domain $\Omega \subset \Re^3$
  let $p_h \in \calP_1(\calT_h) \cap \Honeo$ be piecewise
  linear on $\calT_h$ and let $S_h = W(\nabla p_h)$. Then $S_h \in
  RT_1(\calT_h)$ and is skew.
  
  The PEERS space has $\bbG^{skw} = \calP_1^{cont}(\calT_h)^{d
    \times d}_{skw}$, so $S_h \in Z^{sym}$ if
  $$
  0 
  = -\int_\Omega \nabla p_h . v_h
  = \int_\Omega  p_h  \dv(v_h)
  \qquad v_h \in \calP_1^{cont}(\calT_h)^d.
  $$
  Notice that $\dv(v_h) \subset \calP_0(\calT_h)$ and this
  later space has dimension equal to the number of tetrahedra
  in $\calT_h$ which we denote by $t$. The inf--sup condition
  will then fail if we show that the dimension of 
  $\calP_1(\calT_h) \cap \Honeo$, namely the number of
  internal vertices of $\calT_h$, is larger than $t$.

  Recall that Euler's formula states
  $
  t - f + e - v = O(1),
  $
  where $t$, $f$, $e$, and $v$, denote the number of tetrahedra,
  triangular faces, edges, and vertices of $\calT_h$. Since each
  tetrahedron has four faces and each (interior) face is the 
  intersection of two tetrahedra it follows that
  $4t \simeq 2f$. Similarly, $2 e = \dbar v$ where
  $\dbar$ is the average degree of the vertices in $\calT_h$. It
  follows that
  $$
  t \simeq v - e = (1 - 2/\dbar) v.
  $$
  This formula is asymptotically correct for large $v$ since
  the boundary contains $O(v^{2/3})$ vertices. Thus for large
  meshes the skew subspace of $Z_h^{sym}$ has dimension at
  least $O((2/\dbar) v)$.
\end{example}

\subsubsection{Augmented AFW Element} \label{sec:AFW} In this section
we augment the two dimensional Arnold-Falk-Winther element
\cite{arfawi07} with a suitable class of symmetric matrices to obtain
a triple satisfying Assumption \ref{ass:infsup}.  A counterexample
shows that this construction fails in three dimensions.

\begin{lemma} \label{lem:bdm} Let $\calT_h$ be a triangulation of a
  bounded Lipschitz domain $\Omega \subset \Re^2$ and let
  \begin{align*}
    \bbG_h &= \bbG \cap \left(\calP_0^{disc} (\calT_h)^{2 \times 2}_{skw}
    \oplus \calP_1^{disc}(\calT_h)^{2 \times 2}_{sym} \vph\right), \\
    \bbU_h &= \calP_0^{disc}(\calT_h)^2, \\
    \bbS_h &=  \bbS \cap BDM_1(\calT_h). 
  \end{align*}
  Then the triple $\bbG_h \times \bbU_h \times \bbS_h$ satisfies 
  Assumption \ref{ass:infsup} with constant depending only upon
  the aspect ratio of $\calT_h$.
\end{lemma}
\begin{remark}
  The AFW finite element space is $(\bbG\cap \calP_0^{disc}
  (\calT_h)^{2 \times 2}_{skw})\times \bbU_h \times \bbS_h$.
\end{remark}
\begin{proof}
  Let the macroelements consist of a non-boundary triangle and
  the three triangles adjacent to it. On each triangle the functions
  in $\bbS_h$ are piecewise linear, so it is immediate that the
  functions orthogonal to $\calP_1^{disc}(\calT_h)^{2 \times
    2}_{sym}$ take the form
  $$
  S_h = \delta I + W(a)
  \qquad \delta, a \in \calP_1^{disc}(M).
  $$
  Arguing as in steps (1) and (2) above, it follows that $a$ and $\delta$
  are smooth, so they must be linear polynomials on $M$. If, in addition,
  $S_h$ is orthogonal to $\bbG_h^{skw}$, it follows that
  $$
  0 = \int_K a_0 + a_1 x + a_2 y = a_0 + (a_1,a_2).\xbar_K,
  \qquad K \subset M,
  $$
  where we have written $a(x,y) = a_0 + a_1 x + a_2 y$, and $\xbar_K$
  denotes the centroid of $K$. If $a(x,y)$ is non-zero it follows that
  the four centroids of the triangles $K \subset M$ lie on the line $0
  = a_0 + a_1 x + a_2 y$ which is impossible. A proof of this
  intuitively obvious geometric property is given in the Appendix.
\end{proof}

\begin{example}
  Given a triangulation $\calT_h$ of a domain $\Omega \subset \Re^3$,
  let $p_h \in \calP_2^{cont}(\calT_h)$ be piecewise
  quadratic and let $S_h = W(\nabla p_h)$. Then $S_h \in
  BDM_1(\calT_h)$ and is skew.
  
  The augmented AFW space has $\bbG^{skw} = \calP_0^{disc}(\calT_h)^{d
    \times d}_{skw}$, so $S_h \in Z^{sym}$ if
  $$
  0 
  = -\int_\Omega \nabla p_h . u_h
  = \sum_k \int_k p_h  [u_h.n] 
  \qquad u_h \in \calP_0^{disc}(\calT_h)^d,
  $$
  where the sum is over the (triangular) faces in $\calT_h$ and $n$
  denotes their normal. In this formula $[u_h.n] \equiv u_h.n$ for faces on
  the boundary.

  If $k$ is a triangle and the mid points of the edges are denoted by
  $\{\xbar_1^k, \xbar_2^k, \xbar_3^k\}$, then the quadrature rule 
  $$
  \int_k f = (|k|/3) \left(f(\xbar^k_1) 
    + f(\xbar^k_2) + f(\xbar^k_3) \vph\right)
  $$
  is exact on $\calP_2(k)$. It follows that $S_h \in Z_h^{sym}$ if
  $p_h(\xbar_i) = 0$ at the mid points of the edges in $\calT_h$.
  Upon recalling that the degrees of freedom for the piecewise
  quadratic finite element space are the function values at the
  vertices and at the mid points of the edges, it follows that
  the skew subspace of $Z^{sym}_h$ has dimension at least as
  large as the number of vertices in $\calT_h$.
\end{example}

\subsection{A New Family of Elements in 2 and 3 Dimensions}\label{sec:fespaces3d}
In this section we construct a family of composite elements that satisfy
Assumption \ref{ass:infsup}, using the div-stable elements of Scott and
Vogelius \cite{scvo85rairo,Zh05}. We make use of the following result which mirrors results shown in \cite[Theorem 9.1]{boffibook} for two dimensions and in
\cite[Proposition 4]{bobrfo09} when $d=3$.

\begin{theorem} \label{thm:bb} Let $(\bbU_h, \bbS_h) \subset \Lfour^d
  \times \bbS$ be a div-stable pair of spaces,
  $$
  \inf_{u_h \in U_h} \sup_{S_h \in \bbS_h} 
  \frac{(u_h, \dv(S_h))}{\normS{S_h} \lfour{u_h}} \geq c
  \quad \text{ and } \quad
  \dv(\bbS_h) \subset \bbU_h.
  $$
  If $V_h^d \times P_h \subset \Honeo^d \times \Ltwoo$ is a div-stable
  velocity--pressure space for the Stokes problem and $\Curl(V_h)
  \subset \bbS_h$, then $W(P_h) \times \bbU_h \times \bbS_h$ is a
  stable triple for the elasticity problem with weak symmetry.
\end{theorem}

Augmenting the spaces constructed in this theorem using Raviart--Thomas
and Scott--Vogelius elements gives a family of spaces satisfying
Assumption \ref{ass:infsup}.

\begin{lemma} \label{lem:SV}
  Let $k\ge 1$ be an integer and let $\calT_h$ be a triangulation of a
  bounded Lipschitz domain $\Omega \subset \Re^d$, $d=2$ or $d=3$,
  and let $\calT_h^r$ denote the barycentric refinement of $\calT_h$ and
  \begin{align*}
    \bbG_h &= \bbG\cap\calP_k^{disc} (\calT_h^r)^{d \times d}, \\
    \bbU_h &= \calP_k^{disc}(\calT_h^r)^d, \\
    \bbS_h &= \bbS \cap RT_k(\calT_h^r).
  \end{align*}
  If $k \geq 1$ when $d=2$ or $k \geq 2$ when $d=3$ 
  the triple $\bbG_h \times \bbU_h \times \bbS_h$ satisfies 
  Assumption \ref{ass:infsup} with constant depending only upon
  the aspect ratio of $\calT_h$.
\end{lemma}

\begin{proof}
  Under the assumptions stated on $k$, the Scott--Vogelius space
  $(V_h^d, P_h) \equiv \calP^{cont}_{k+1}(\calT_h^r)^d \times
  \calP_k^{disc}(\calT_h^r)$ is a div--stable element for the Stokes
  problem \cite{Zh05}; moreover, $\Curl(V_h) \subset RT_k(\calT_h^r)$ since
  (i) functions in $\Curl(V_h)$ belong to $\Hdiv$, and (ii) the
  Raviart--Thomas spaces contain all piecewise polynomials of degree
  $k$ in $\Hdiv$. It follows that $W(P_h) \times \bbU_h \times \bbS_h$
  is a stable triple for the elasticity problem with weak symmetry.

  Upon recalling that the divergence free functions in $RT_k$
  are piecewise polynomials of degree $k$, it follows 
  that functions in $Z^{sym}$ are symmetric (pointwise), and
  the inf-sup condition in Assumption \ref{ass:infsup} follows
  upon setting $G_h^{sym}$ to be the trace-free part of $S_h$ and using
  Lemma \ref{lem:trace} to bound the trace.
\end{proof}

Condensing out internal degrees of freedom from composite elements
significantly reduces the number of unknowns. The following example
illustrates this for the lowest order two dimensional element.

\begin{figure}[!ht]
\centerline{
\subfigure[$\bbS_h$ on $K^r$]
 {\label{fig:rt1comp}
 \includegraphics[width=1.25in]{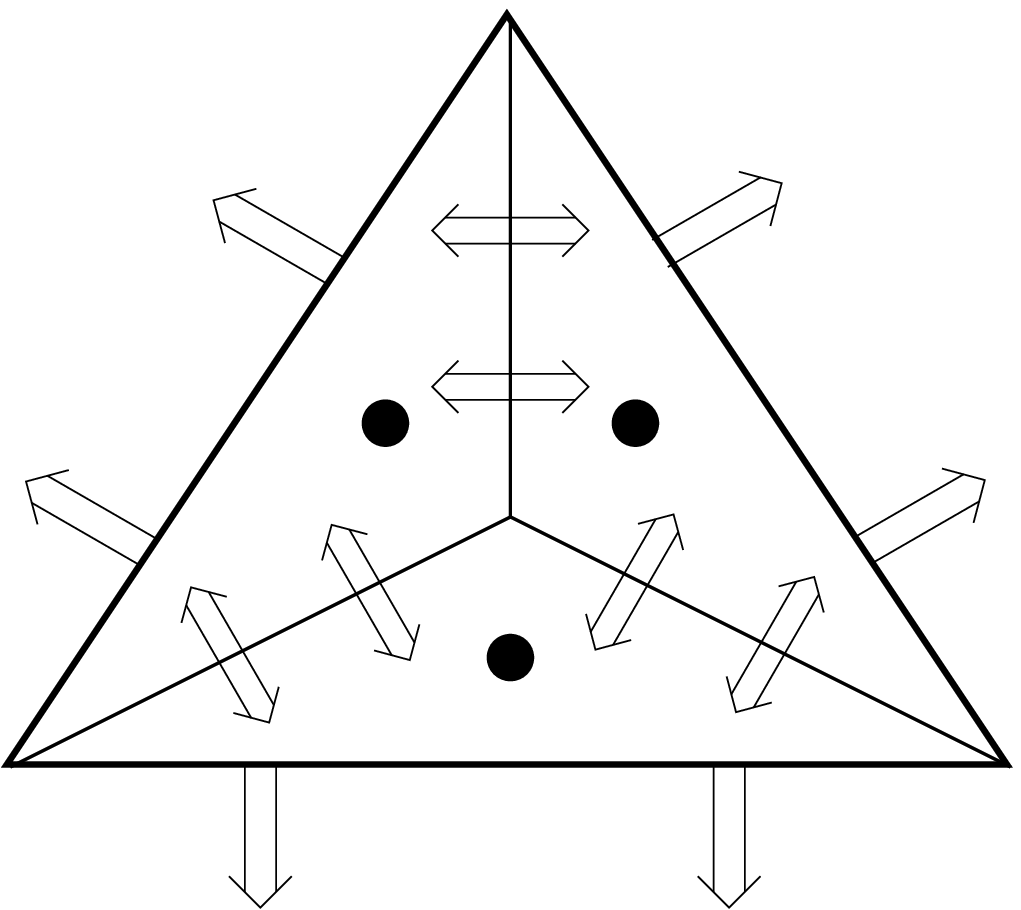}
 }\hskip 1in
\subfigure[Condensed $\bbS_h$]
 {\label{fig:rt1cond}
 \includegraphics[width=1.25in]{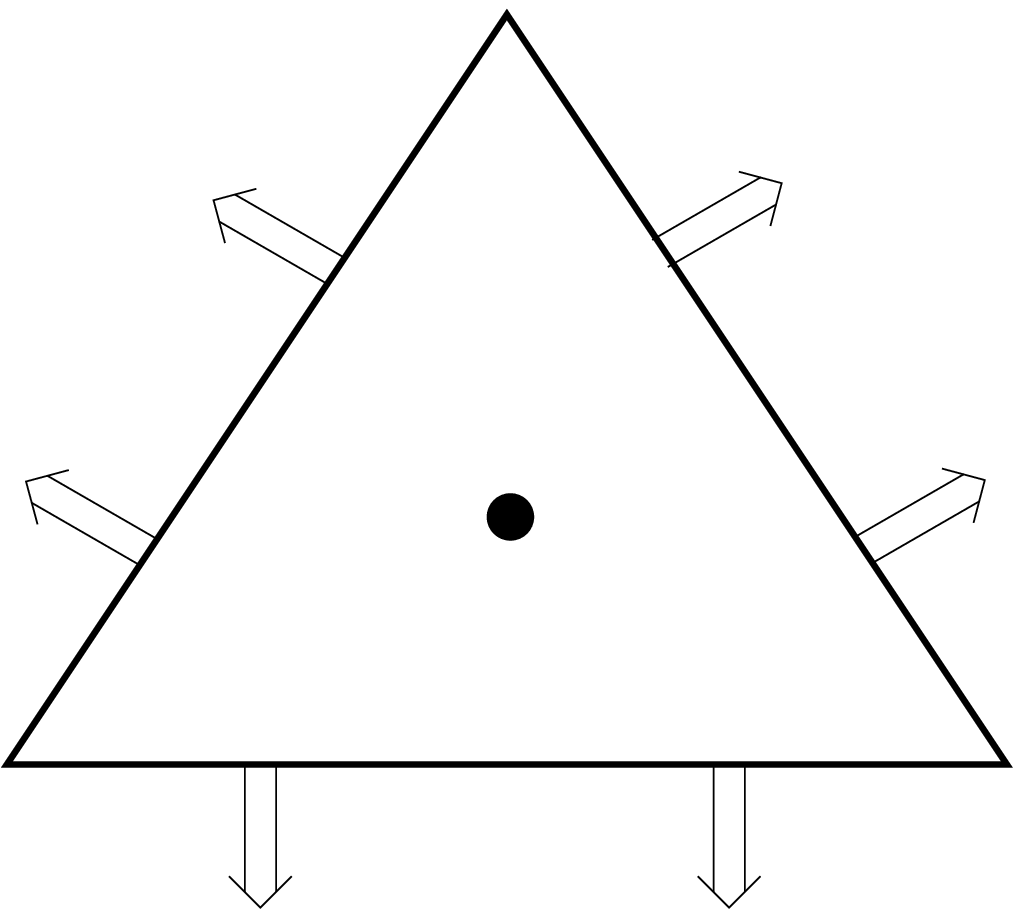}
 }}
\caption{Degrees of freedom for the lowest order ($k=1$)
  element in two space dimensions, which consists of two rows, each of
  which are first-order Raviart--Thomas vectors.  The double arrows
  represent continuity of the normal components of each row of $S_h$,
  and the dot represents the four internal degrees of freedom of $S_h$
  (two per row).  The diagram on the left represents $\bbS_h$ on a
  barycentric refined mesh, and the diagram on the right represents the
  condensed element.}\label{fig:rt1}
\end{figure}

\begin{example}
If $K \in \calT_h$ is a triangle and $K^r$ is its barycentric
refinement, the lowest order two dimensional element ($k=1$) would have
(see Figure \ref{fig:rt1comp})
$$
dim(\bbG_h(K^r)) = 27, \quad
dim(\bbU_h(K^r)) = 18, \qquad
dim(\bbS_h(K^r)) = 36.
$$
Consider then the two subspaces of weakly symmetric tensors on $K^r$
\begin{align*}
  \bbS^{sym}_h(K^r) &= \{S_h \in \bbS_h(K^r) \sst 
  \int_K S_h:G_h = 0, \,\, G_h \in \bbG_h(K^r)^{skw} \}, \\
  \bar{\bbS}_h(K) &= \{S_h \in \bbS_h^{sym}(K^r) \sst 
  \dv(S_h) \in \calP_1(K) \}.
\end{align*}
Then $\bar{\bbS}_h(K)$ has dimension 15 and a set of degrees of
freedom almost identical to to $RT_1(K)$; namely, the trace of $Sn$ on
each edge of $K$ and the average of the {\em symmetric} part of $S_h$
over $K$ (the average of the skew part being zero), see Figure
\ref{fig:rt1cond}. This gives rise to a decomposition
$$
\bbS^{sym}_h(K^r) = \bar{\bbS}_h(K) + \bbS^0_h(K^r),
\qquad
\bbS^0_h(K^r) = \{S_h \in \bbS_h(K^r) \sst S_h n = 0 \text{ on }
\partial K\}.
$$
The degrees of freedom for $\bbS^0_h(K^r)$ would be the ones illustrated
in Figure \ref{fig:rt1comp} which are interior to $K$. 

Let
$$
\bar{\bbU}_h =  \calP_1^{disc}(\calT_h)^2
\quad \text{ and } \quad
\bar{\bbS}_h = \bbS \cap \{S_h \in \bbS \sst S_h|_K \in \bar{\bbS}_h(K),
\,\, K \in \calT_h\}.
$$
Then solutions of the Navier--Stokes problem would seek
$(\ubar_h, \Sbar_h) \in \bar{\bbU}_h \times \bar{\bbS}_h$ such
that
\begin{gather*}
\int_\Omega G_h u_h. \vbar_h - \dv(\Sbar_h + S^0_h). \vbar_h 
= \int_\Omega f.\vbar_h,
\qquad \vbar_h \in \bar{\bbU}_h \\
\int_\Omega G_h:\Tbar_h + \ubar_h. \dv(\Tbar_h) = 0,
\qquad \Tbar_h \in \bar{\bbS}_h,
\end{gather*}
where on each element
$$
(G_h, u_h, S^0_h)
\in 
\calP_1^{disc}(K^r)^{2 \times 2} \times 
\calP_1^{disc}(K^r)^2 \times 
\bbS^0_h(K^r)
$$
are determined from $(\ubar_h,\Sbar_h)$ as the solution of the local problem: 
\begin{gather*}
\int_K \calA(G_h):H_h - (1/2)(u_h \otimes u_h):H_h
- (\Sbar + S^0_h):H_h = 0, 
\qquad H_h \in \calP_1^{disc}(K^r)^{2 \times 2} \\
\int_K G_h u_h . v_h - \dv(\Sbar_h + S^0_h). v_h 
= \int_\Omega f .v_h,
\qquad v_h \in \calP_1^{disc}(K^r)^2 \\
\int_K G_h:T^0_h + u_h. \dv(T^0_h) = 0,
\qquad T^0_h \in \bbS^0_h(K^r), \\
\int_K u_h.\vbar_h = \int_K \ubar_h.\vbar_h, 
\qquad \vbar_h \in \calP_1^{disc}(K^r)^2.
\end{gather*}
The last equation can be eliminated if a basis for the orthogonal
decomposition $\calP_1^{disc}(K^r)^2 = \calP_1^{disc}(K)^2
\oplus (\calP_1^{disc}(K)^2)^\perp$ is available.
\end{example}

\section{Numerical Examples} \label{sec:numerics}
The following non-homogeneous solution of the two dimensional Navier--Stokes equations is the stationary analog of the solution from
\cite{sh93}.
\begin{align*}
u &= \left( (-m/k) \sin(kx)\cos(my), \cos(kx) \sin(my) \vph\right)^T \\
p &= (-1/2) \left( |u|^2 + (1-(m/k)^2) \sin^2(kx) \sin^2(my) \vph\right) 
\end{align*}
with right hand side $f = \nu(k^2+m^2) u$. The computational domain
was chosen to be $\Omega = (-1,1)^2$, the traction boundary
condition was specified on the right edge ($x=1$), and Dirichlet data was specified on the
remainder of the boundary.  Triangulations were formed by sub-dividing
the square uniformly into squares of edge length $h=2/N$ and dividing
each of these into two triangles. The parameters were selected to be $k
= \pi$, $m = \pi/2$, and $\nu = 1/20$.

Figures \ref{tbl:NavierStokesAFW} and \ref{tbl:NavierStokesSV} tabulate
the $L^2(\Omega)$ errors of the approximate solutions computed using the dual--mixed
formulation with the augmented AFW and new element respectively. A first
order rate of convergence for the augmented AFW element is observed and
a second order rate for the new element is achieved for the finer meshes.

\begin{table}
  \small
  \begin{center}
  \begin{tabular}{|c|ccccc|} \hline
    $h$ & $G^{sym}$ & $G^{skw}$ & $u$ & $S$ & $\dv(S)$ \\ \hline
    $1/4 $ & 
    6.883930e-01 & 6.544852e-01 & 2.312414e-01 & 1.505661e-01 & 2.405736e-01 \\
    $1/8 $ & 
    3.314210e-01 & 3.269643e-01 & 1.157281e-01 & 6.841425e-02 & 1.205427e-01 \\
    $1/16$ & 
    1.637091e-01 & 1.631461e-01 & 5.785624e-02 & 3.320099e-02 & 6.013830e-02 \\
    $1/32$ & 
    8.157049e-02 & 8.150047e-02 & 2.892592e-02 & 1.646940e-02 & 3.004137e-02 \\
    $1/64$ & 
    4.074540e-02 & 4.073677e-02 & 1.446263e-02 & 8.218179e-03 & 1.501600e-02 \\
    $1/128$& 
    2.036695e-02 & 2.036590e-02 & 7.231270e-03 & 4.107087e-03 & 7.507253e-03
    \\ \hline
    Norm & 2.776802 & 2.776802 & 1.118034 & 0.905688 & 0.927988  \\ \hline
    Rate &  1.0135 & 1.0013  & 0.9999  & 1.0333  & 1.0008   \\ \hline
  \end{tabular}
  \end{center}
\normalsize
  \caption{$\Ltwo$ errors for the dual--mixed formulation of the Navier--Stokes 
  problem using the augmented AFW element.}
  \label{tbl:NavierStokesAFW}
\end{table}

\begin{table}
\small
\begin{center}
\begin{tabular}{|c|ccccc|} \hline
  $h$ & $G^{sym}$ & $G^{skw}$ & $u$ & $S$ & $\dv(S)$ \\ \hline
  $1/4 $ & 
  2.451267e-01 & 4.576332e-01 & 2.399623e-02 & 3.780692e-02 & 9.647025e-02 \\
  $1/8 $ & 
  8.182080e-02 & 2.012711e-01 & 5.839852e-03 & 1.210886e-02 & 4.150840e-02 \\
  $1/16$ & 
  2.414153e-02 & 7.073387e-02 & 1.284591e-03 & 3.530203e-03 & 1.453778e-02 \\
  $1/32$ & 
  6.452742e-03 & 2.032373e-02 & 2.863796e-04 & 9.407521e-04 & 4.194951e-03 \\
  $1/64$ & 
  1.650640e-03 & 5.324535e-03 & 6.817378e-05 & 2.404370e-04 & 1.103307e-03 \\
  $1/128$ & 
  4.159781e-04 & 1.352113e-03 & 1.679452e-05 & 6.057340e-05 & 2.807632e-04 
  \\ \hline
  Norm & 2.776802 & 2.776802 & 1.118034 & 0.905688 & 0.927988  \\ \hline
  Rate & 1.9543   & 1.9060   & 2.0842   & 1.9563   & 1.9010  \\ \hline
\end{tabular}
\end{center}
\normalsize
\caption{$\Ltwo$ errors for the dual--mixed formulation of the Navier--Stokes problem
  using the new element (rate for $h \in \{ 1/16, 1/32, 1/64, 1/128\}$).}
\label{tbl:NavierStokesSV}
\end{table}

\def\ocirc#1{\ifmmode\setbox0=\hbox{$#1$}\dimen0=\ht0 \advance\dimen0
  by1pt\rlap{\hbox to\wd0{\hss\raise\dimen0
  \hbox{\hskip.2em$\scriptscriptstyle\circ$}\hss}}#1\else {\accent"17 #1}\fi}

\appendix

\section{Collinearity of Triangle Centroids}

The following lemma was used in the proof of Lemma \ref{lem:bdm}.

\begin{lemma} Let $\{K_i\}_{i=0}^3$ be triangles in the plane with
  disjoint interiors and let each of $K_1$, $K_2$, and $K_3$ have an
  edge in common with $K_0$. Then the centroids of the four triangles
  are not collinear.
\end{lemma}

\begin{figure}
\begin{center}
\psfrag{vv}{$3v$}
\psfrag{ww}{$3w$}
\psfrag{ll}{$\ell$}
\includegraphics[height=2.0in, angle=0]{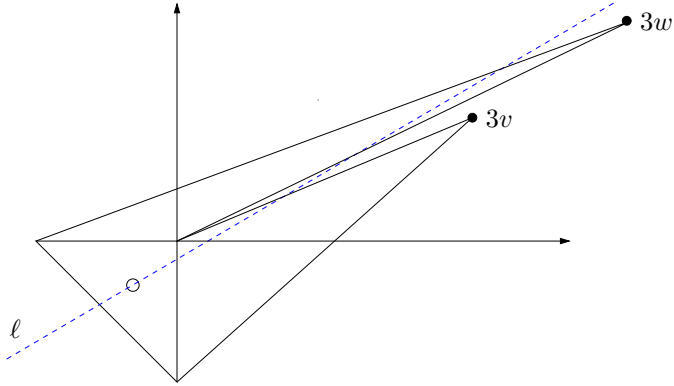}
\caption{Centroids of the three triangles can not lie on the dashed line.} 
\label{fig:infSup2D}
\end{center}
\end{figure}

\begin{proof} To obtain a contradiction, let $\ell$ be a line
  containing all four centroids, then two of the three centroids of $K_1$,
  $K_2$ and $K_3$ lie on one side of the centroid of $K_0$. Since
  averages map to averages under affine maps, it suffices to consider
  the situation where $K_0$ is the triangle with coordinates $(0,
  -3e_1, -3e_2)$ and and the centroids of the triangles sharing the
  top and right edges of $K_0$ have their centroids on the same side
  of $\ell$. Assume without loss of generality that $\ell$ exits $K_0$
  on the right, so that it has slope less than $1$ (otherwise reflect
  about the line $y=x$).

  Let the top triangle have vertex $3w$ and the triangle on the right have
  vertex $3v$, so that the centroids are
  $$
  c_{top} = (1/3)(0 + 3w - 3e_1) = w-e_1,
  \quad \text{ and } \quad
  c_{right} = (1/3)(0 + 3v - 3e_2) = v-e_2.
  $$
  Since the top triangle (a) has all its vertices above the $x$--axis
  and (b) intersects $\ell$, it follows that $w$ is in the positive
  quadrant, as shown in Figure \ref{fig:infSup2D}.
  
  If the centroids $c_{top}$, $c_{right}$ and $c = (-1,-1)$ of $K_0$ are
  collinear, there exists $\lambda > 0$ such that the equation
  $$
  \lambda (c_{top} - c) = c_{right} - c,
  \quad \text{ i.e. } \quad
  \lambda (w-e_1 - c) = v-e_2 - c,
  $$
  has a solution with $w$ in the positive quadrant, $\lambda > 0$ and
  $v_1 > 0$.  Writing out the two components of this shows
  $$
  \left.
    \begin{array}{c} 
      \lambda w_1  = v_1 + 1 \\ \lambda (w_2 + 1)  = v_2
    \end{array}
  \right\} \quad \Rightarrow \quad
  w_1 v_2 = (v_1+1)(w_2+1).
  $$
  Note that $v_1 > 0$ and the second equation shows $v_2 > 0$, so $v$ is
  also in the positive quadrant. 
  
  Next, geometric consistency (triangle interiors do not intersect)
  requires that the angle $w$ forms with the $x$-axis to be greater than the
  angle $v$ forms;
  $$
  w_2 / w_1 \geq v_2 / v_1,
  \quad \text{ or } \quad
  v_1 w_2 \geq w_1 v_2.
  $$
  Expanding the equation $w_1 v_2 = (v_1+1)(w_2+1)$ shows
  $$
  w_1 v_2 - v_1 w_2 = 1 + v_1 + w_2,
  $$
  and the inequality $0 \geq w_1 v_2 - v_1 w_2$ shows no solution exits.
\end{proof}

\end{document}